\documentclass[12pt]{amsart}

\usepackage{a4wide}
\usepackage{amsmath, amssymb}
\usepackage{amsthm}
\usepackage[utf8]{inputenc}
\usepackage{subfigure}
\usepackage{enumerate}
\usepackage{cancel}
\usepackage{xspace}
\usepackage{float}
\usepackage{mathrsfs}
\usepackage[colorlinks, citecolor=blue, anchorcolor=blue, linkcolor=blue, urlcolor=blue]{hyperref}
\usepackage{color,xcolor}

\usepackage{wrapfig}

\usepackage{tikz}
\usetikzlibrary{patterns}

\usepackage{algorithm}
\usepackage{algorithmic}

\def\SYoung#1{\vbox{\smallskip\offinterlineskip
    \halign{&\vbox{##}\kern-\SThickness\cr #1}}}

\newdimen\SSquaresize \SSquaresize=4.5pt
\newdimen\SThickness \SThickness=.15pt
\newdimen\SCorrection \SCorrection=7pt

\def\SCarre#1{\hbox{\vrule width \SThickness
   \vbox to \SSquaresize{\hrule height \SThickness\vss
      \hbox to \SSquaresize{\hss$\scriptstyle#1$\hss}
   \vss\hrule height\SThickness}
   \unskip\vrule width \SThickness}
   \kern-\SThickness}

\allowdisplaybreaks

\makeatletter \@addtoreset{equation}{section}

\newtheorem{theorem}{Theorem}[section]
\newtheorem{proposition}[theorem]{Proposition}
\newtheorem{lemma}[theorem]{Lemma}

\newtheorem{conjecture}[theorem]{Conjecture}

\title[Analytic properties of Speyer's $g$-polynomial]{Analytic properties of Speyer's $g$-polynomial of uniform matroids}

\author[R. Zhang]{Rong Zhang}
       \address{Department of Mathematics and Physics, Luoyang Institute of Science and Technology, Luoyang 471023, P. R. China.}
       \email{zhr168@lit.edu.cn}

\author[James J. Y. Zhao]{James Jing Yu Zhao}
       \address{School of Accounting, Guangzhou College of Technology and Business,
       Foshan 528138, P. R. China.}
       \email{zhao@gzgs.edu.cn}



\begin{document}

\begin{abstract}
Let $U_{n,d}$ denote the uniform matroid of rank $d$ on $n$ elements.
We obtain some recurrence relations satisfied by Speyer's $g$-polynomials $g_{U_{n,d}}(t)$ of $U_{n,d}$.
Based on these recurrence relations,
we prove that the polynomial $g_{U_{n,d}}(t)$ has only real zeros for any $n-1\geq d\geq 1$.
Furthermore, we show that the coefficient of $g_{U_{n,[n/2]}}(t)$ is asymptotically normal by local and central limit theorems.
\end{abstract}

\maketitle

\section{Introduction}\label{sec:1}

In the study of the $f$-vector conjecture of tropical linear spaces, Speyer \cite{Speyer2009} discovered an invariant, denoted by $g_M(t)$,
for a (loopless and coloopless) $\mathbb{C}$-realizable matroid $M$.
This invariant has been used by Speyer to prove bounds on the complexity of Kapranov’s Lie
complexes \cite{Kapranov1993}, Hacking, Keel and Tevelev’s very stable pairs \cite{HKT2006}, and tropical linear spaces \cite{Speyer2008}. It turns out that this invariant is of independent combinatorial interest.
In a subsequent work of Fink and Speyer \cite{FinkSpeyer2012}, the invariant $g_M(t)$ was defined for
an arbitrary (loopless and coloopless) matroid $M$.
L\'opez de Medrano, Rinc\'on, and Shaw \cite{LRS2020} conjectured a Chow-theoretic formula for Speyer's $g$-polynomial,
which was later proved by Berget, Eur, Spink and Tseng \cite{BEST-2023}.
Ferroni \cite{Ferroni} provided a combinatorial way of computing $g_M(t)$ for an arbitrary Schubert matroid $M$.
Ferroni and Schr\"{o}ter \cite{Ferroni-Schroter-2023} proved that the coefficients of $g_M(t)$ are nonnegative for any sparse paving matroid $M$.

It is worth mentioning that the $g$-polynomial of uniform matroids has been computed by Speyer \cite{Speyer2009},
and this polynomial plays an important role in his proof of the $f$-vector conjecture in the case of a tropical linear
space realizable in characteristic $0$.
Let $U_{n,d}$ denote the uniform matroid of rank $d$ on $n$ elements. Speyer \cite{Speyer2009} obtained the following result.

\begin{proposition}[{\cite[Proposition 10.1]{Speyer2009}}]
For any $n\geq 2$ and $1\leq d\leq n-1$, the $g$-polynomial of the uniform matroid $U_{n,d}$ is given by
\begin{align}\label{g-uni-dn(t)}
 g_{U_{n,d}}(t)=\sum_{i=1}^{\min(d,n-d)} \frac{(n-i-1)!}{(d-i)!(n-d-i)!(i-1)!}\, t^i.
\end{align}
\end{proposition}

The main objective of this paper is to study the analytic properties of $g_{U_{n,d}}(t)$. For notational convenience, let $g_{n,d}(t)=g_{U_{n,d}}(t)$ and $S_i(n,d)=[t^i]g_{n,d}(t)$, the coefficient of $t^i$ in $g_{n,d}(t)$. From \eqref{g-uni-dn(t)} it immediately follows that $g_{n,d}(t)= g_{n,n-d}(t)$. Thus, from now on, we may assume that $1\le d\le \lfloor n/2 \rfloor$, in which case we have $\deg g_{n,d}(t)=d$.

The remainder of this paper is organized as follows. In Section \ref{Sec:2} we first obtain a series of recurrence relations satisfied by Speyer's $g$-polynomials $g_{n,d}(t)$ by using the extended Zeilberger algorithm developed by Chen, Hou and Mu \cite{CHM-2012}.
In Section \ref{Sec:3} we prove that $g_{n,d}(t)$ has only real zeros for any $n\geq 2$ and $1\leq d\leq n-1$ by using a result due to Liu and Wang \cite{LiuWang}. Finally, in Section \ref{Sec:4}, we show that the coefficient of Speyer's $g$-polynomial $g_{n,[n/2]}(t)$ is asymptotically normal by local and central limit theorems based on a criterion given by Bender \cite{Bender} and Harper \cite{Harper}.

\section{Recurrence Relations}\label{Sec:2}

The aim of this section is to show some recurrence relations for Speyer's $g$-polynomial $g_{n,d}(t)$. These recurrences will play a key role in the proofs of the interlacing properties and the asymptotic normality presented in Sections \ref{Sec:3} and \ref{Sec:4} respectively.
On one hand, we fix $d\geq 1$ and consider the recurrence relation of the sequence $\{g_{n,d}(t)\}_{n\geq d+1}$.
We have the following result.
\begin{theorem}\label{thm:rec-g-poly}
For $d\ge 1$ and $n\ge d+2$, we have
\begin{align}
 g_{n,d}(t)
=&\, \frac{(2d-n+1)t+n-2}{n-d-1} g_{n-1,d}(t)
 +\frac{t(n-d-2)}{n-d-1} g_{n-2,d}(t), \label{rec:g-p-n1}\\[5pt]
 g_{n,d}(t)
=&\, \frac{dt+n-1}{n-d-1} g_{n-1,d}(t)
 -\frac{t(t+1)}{n-d-1} g_{n-1,d}'(t), \label{rec:g-p-n2}
\end{align}
with initial values $g_{d,d}(t)=0$, $g_{d+1,d}(t)=t$ and $g_{d+2,d}(t)=dt+(d-1)t^2$.
\end{theorem}

On the other hand, we fix $n\geq 3$ and consider the recurrence relation of the sequence $\{g_{n,d}(t)\}_{1\leq d\leq n-1}$.
We have the following result.
\begin{theorem}\label{thm:rec-g-poly2}
For any $n\geq 3$ and $1\leq d\leq n-1$, we have
\begin{align}
 g_{n,d}(t)
=&\, \frac{c_d(n+1)\big(c_d(n)c_d(n+2)t+(n-d)^2+(d-2)^2+n-2\big)}
{(d-1)(n-d)c_d(n+2)} g_{n,d-1}(t)\label{rec:g-p-d1}\\
&\quad
 -\frac{(d-2)(n+1-d)c_d(n)}{(d-1)(n-d)c_d(n+2)} g_{n,d-2}(t),
 \qquad \mbox{ for } 2\le d\le [n/2]+1 \mbox{ and } n\ge 3, \nonumber\\[5pt]
 g_{n,d}(t)
=&\, \frac{(n+1-2d)(n+1-d)t+(n-d)^2+n-2d+1}{(d-1)(n-d)} g_{n,d-1}(t)\label{rec:g-p-d2}\\
&\quad
 -\frac{t(t+1)(n+1-2d)}{(d-1)(n-d)} g_{n,d-1}'(t),
 \qquad \mbox{ for } 2\le d\le n-1 \mbox{ and } n\ge 3, \nonumber
\end{align}
where $c_d(n)=n+1-2d$, $g_{n, 0}(t)=0$, $g_{n, 1}(t)=t$ and $g_{n, 2}(t)=(n-2)t+(n-3)t^2$ for $n\ge 3$.
\end{theorem}

We shall prove Theorems \ref{thm:rec-g-poly} and \ref{thm:rec-g-poly2} by using the extended Zeilberger algorithm developed by Chen, Hou and Mu \cite{CHM-2012} on the basis of Zeilberger algorithm \cite{Zeilberger1991, WZ1992}.
For  a nice survey on the method of creative telescoping and related open problems, see Chen and Kauers \cite{ChenKauers}.
Before proving Theorems \ref{thm:rec-g-poly} and \ref{thm:rec-g-poly2}, let us have a brief overview of the extended Zeilberger algorithm.
Suppose that $f_{1}(k,\,p_1,\,p_2,\,\ldots,\,p_{s}),\,\ldots,\, f_{m}(k,\,p_1,\,p_2,\,\ldots,\,p_{s})$ are $m$ hypergeometric terms of $k$ with parameters $p_1,\,p_2$, $\ldots,\,p_{s}$ such that both
$$
\frac{f_{i}(k,\,p_1,\,p_2,\,\ldots,\,p_{s})}{f_{j}(k,\,p_1,\,p_2,\,\ldots,\,p_{s})}
\quad
{\rm and}
\quad
\frac{f_{i}(k+1,\,p_1,\,p_2,\,\ldots,\,p_{s})}{f_{i}(k,\,p_1,\,p_2,\,\ldots,\,p_{s})}
$$
are all rational functions of $k$ and $p_1,\,p_2,\,\ldots,\,p_{s}$ for any $1\le i,j\le m$. This algorithm is devised to find a hypergeometric term $h(k,\,p_1,\,p_2,\,\ldots,\,p_{s})$
and polynomial coefficients $a_{i}(p_1,\,p_2,\,\ldots,\,p_{s})$ for $1\le i\le m$ which are independent of $k$ such that
\begin{align}\label{eq-telescope}
a_{1}f_1(k)+a_2 f_2(k)+\cdots+ a_{m}f_{m}(k)=u(k+1)-u(k),
\end{align}
where we omit the parameters $p_1,\,p_2,\,\ldots,\,p_{s}$ for brevity.
If we let
$F_i=\sum_k f_i(k)$
for $1\leq i\leq m$ and sum \eqref{eq-telescope} over $k$, then we get a homogeneous relation
\begin{align}\label{relation-desired}
a_1 F_1 +a_2 F_2 +\cdots+a_{m} F_{m}=0.
\end{align}

Hou \cite{APCI} implemented the extended Zeilberger algorithm as the function {\tt Ext\_Zeil} in the Maple package {\tt APCI}.
The calling sequence of this function  is of the form {\tt Ext\_Zeil([$f_1,\, f_2,\,\ldots,f_{m}$]{\rm ,} $k$)}, and the output is of the form $[C,Ca_2/a_1,Ca_3/a_1,\ldots,Ca_{m}/a_1]$, where $C$ is a $k$-free non-zero constant, provided that the algorithm is applicable.

We are now in a position to give a symbolic proof of Theorems \ref{thm:rec-g-poly} and \ref{thm:rec-g-poly2}.

\begin{proof}[Proof of Theorems \ref{thm:rec-g-poly} and \ref{thm:rec-g-poly2}]
Let us first show the recursion \eqref{rec:g-p-n1}. To employ the package {\tt APCI}, we first import it in Maple as follows.

\noindent
{\bf\small
 $[>$ with(APCI)};
\vskip -18pt
\begin{align*}
       &\mbox{[{\it AbelZ, Ext\_Zeil, Gosper, MZeil, Zeil, hyper\_simp, hyperterm, poch, qExt\_Zeil},}\\
         &\qquad \mbox{{\it qGosper, qZeil, qbino, qhyper\_simp, qhyperterm, qpoch}]}
\end{align*}
Observe that \eqref{rec:g-p-n1} is of the form \eqref{relation-desired} with
$$
f_1=S_k(n,d) t^k, \quad f_2=S_k(n-1,d) t^k,\quad f_3=S_k(n-2,d) t^k,
$$
where $S_k(n,d)$ is as defined in Section \ref{sec:1}.
In order to prove \eqref{rec:g-p-n1}, we continue the following set of $f_i$:

\noindent
{\bf\small
$[>$ {\mathversion{bold} \hskip -2pt
$f_1 := \displaystyle{\frac{(n-k-1)!}{(d-k)!(n-d-k)!(k-1)!}} t^k$:}}\\

\noindent
{\bf\small
$[>$ {\mathversion{bold} \hskip -2pt
$f_2 := \displaystyle{\frac{(n-k-2)!}{(d-k)!(n-d-k-1)!(k-1)!}} t^k$:}}\\

\noindent
{\bf\small
$[>$ {\mathversion{bold} \hskip -2pt
$f_3 := \displaystyle{\frac{(n-k-3)!}{(d-k)!(n-d-k-2)!(k-1)!}} t^k$:}}\\

Then we run the command of the main function:

\noindent
{\bf\small
$[>$ {\bf\small Ext}$\_${\bf\small Zeil}\hskip -0.1pt {\mathversion{bold}$([f_1,f_2,f_3],k)$;}}\\
\begin{align*}
 \!\! \left[k\_free_1,\ \frac{k\_free_1 (2dt-nt+n+t-2)}{d-n+1},\
  -\frac{t(d-n+2) k\_free_1}{d-n+1} \right].
\end{align*}

The above output implies that there is some nonzero constant $C$ free of $k$ such that
$$
C \cdot g_{n,d}(t) +C\cdot \frac{2dt-nt+n+t-2}{d-n+1}\, g_{n-1,d}(t) + C\cdot \frac{-t(d-n+2)}{d-n+1}\, g_{n-2,d}(t)=0,
$$
which leads to the recursion \eqref{rec:g-p-n1}.
By \eqref{g-uni-dn(t)}, it is easy to check that $g_{d+1,d}(t)=t$ and $g_{d+2,d}(t)=dt+(d-1)t^2$ for $d\ge 1$. Thus, \eqref{rec:g-p-n1} is proved.

The detailed proofs of recurrence relations \eqref{rec:g-p-n2}, \eqref{rec:g-p-d1} and \eqref{rec:g-p-d2} are similar to that of \eqref{rec:g-p-n1}, and hence are omitted here.
\end{proof}

We further obtain the following two triangular recurrence relations.
\begin{theorem}\label{thm:2.2}
For $n\ge 3$ and $2\le d \le n-1$, we have
\begin{align}
 g_{n,d}(t)
=&\, \frac{n-d}{d-1} g_{n,d-1}(t)
 +\frac{t(n+1-2d)}{d-1} g_{n-1,d-1}(t), \label{rec:g-p-nd1}\\
 g_{n,d}(t)
=&\, \frac{(d-1)(n-d)}{(n-2d)(n+1-2d)t+(d-1)^2} g_{n,d-1}(t)\label{rec:g-p-nd2}\\
 &\quad
  +\frac{(n+1-2d)(n-1-d)t}{(n-2d)(n+1-2d)t+(d-1)^2} g_{n-1,d}(t).\nonumber
\end{align}
\end{theorem}

\begin{proof}
By applying the extended Zeilberger algorithm, one can obtain the recursions \eqref{rec:g-p-nd1} and \eqref{rec:g-p-nd2} in a similar argument stated in the proof of Theorems \ref{thm:rec-g-poly} and \ref{thm:rec-g-poly2}.
\end{proof}

We also find the following interesting recurrence relations for Speyer's $g$-polynomials, and omit the proofs. As will be seen, the recursions provided by Theorem \ref{thm:gpn=2d} will lead to some new interlacing properties for Speyer's $g$-polynomials showed in Theorem \ref{thm:intlc-n=2d}.

\begin{theorem}\label{thm:gpn=2d}
For $d\ge 1$, Speyer's $g$-polynomial satisfies the following recurrence relations.
\begin{align}
 g_{2d+2,d+1}(t)
=&\, \frac{(t+2)(2d-1)}{d} g_{2d,d}(t)
  -\frac{(d-1)t^2}{d} g_{2d-2,d-1}(t), \label{rec:g-p-n2d} \\
 g_{2d+3,d+1}(t)
=&\, \frac{2(2d^2t+4d^2-1)}{(d+1)(2d-1)} g_{2d+1,d}(t)
  -\frac{(2d+1)(d-1)t^2}{(d+1)(2d-1)} g_{2d-1,d-1}(t), \label{rec:g-p-n2do}\\
 g_{2d+1,d}(t)
=&\, \frac{2d-1}{d} g_{2d,d}(t) + \frac{(d-1)t}{d} g_{2d-1,d-1}(t),
  \label{rec:g-p-n2d+1}\\
 g_{2d+2,d+1}(t)
=&\, 2\, g_{2d+1,d}(t) + t\, g_{2d,d}(t), \label{rec:g-p-n2d-1}
\end{align}
with the convention that $g_{0,0}(t)=g_{1,0}(t)=0$.
\end{theorem}

To end this section, we show four more recurrence relations for Speyer's $g$-polynomial. The recurrences given by Theorem \ref{thm:rec-f'f''} together with \eqref{rec:g-p-n2d+1} and \eqref{rec:g-p-n2d-1} will be employed in the proof of the asymptotic normality stated in Section \ref{Sec:4}.

\begin{theorem}\label{thm:rec-f'f''}
For any given integer $d\ge 2$, we have
\begin{align}
 g_{2d,d}(t)
=&\, \frac{(d-1)t}{dt+1}\, g_{2d-1,d-1}(t)
 +\frac{t(t+1)}{dt+1}\, g_{2d,d}'(t), \label{rec:g-p-n2d-1'}\\
 g_{2d,d}(t)
=&\, \frac{(2d-1)t+2d}{dt+d+1}\, g_{2d-1,d-1}(t)
  +\frac{t(t+1)^2}{(d-1)(dt+d+1)}\, g_{2d,d}''(t),\label{rec:g-p-n2d-1''}
\end{align}
and for any given integer $d\ge 1$, we have
\begin{align}
 g_{2d+1,d}(t)
=&\, \frac{dt}{(d+1)t+1}\, g_{2d,d}(t)
 +\frac{t(t+1)}{(d+1)t+1}\, g_{2d+1,d}'(t), \label{rec:g-p-n2d+1'}\\
 g_{2d+1,d}(t)
=&\, \frac{2dt+2d+1}{(d+1)t+d+2}\, g_{2d,d}(t)
  +\frac{t(t+1)^2}{d(dt+t+d+2)}\, g_{2d+1,d}''(t).\label{rec:g-p-n2d+1''}
\end{align}
\end{theorem}

\section{Real zeros and interlacing properties}\label{Sec:3}

The objective of this section is to show the real-rootedness and some interlacing properties of Speyer's $g$-polynomial $g_{n,d}(t)$. Consequently, we obtain the log-concavity and unimodality of this polynomial. The main result of this section is as follows.

\begin{theorem}\label{thm:rz-lcv-1}
For any given integers $n\ge 2$ and $1\le d \le n-1$, the Speyer's $g$-polynomial $g_{n,d}(t)$ has only real zeros.
\end{theorem}

Before proving Theorem \ref{thm:rz-lcv-1}, let us first recall some definitions and results on Sturm sequences.
Let PF be the set of real-rooted polynomials with nonnegative coefficients, including any nonnegative constant for convenience.
Given two polynomials $f(t),\,g(t)\in \mathrm{PF}$, assume $f(\alpha_i)=0$ and $g(\beta_j)=0$.
We say that $g(t)$ interlaces $f(t)$, denoted $g(t)\preceq f(t)$, if either
$\deg f(t)=\deg g(t)=n$ and
\begin{align}\label{defi:inter-z1}
{\beta_n\le \alpha_n\le\cdots\le \beta_2\le \alpha_2\le \beta_1\le \alpha_1},
\end{align}
or $\deg f(t)=\deg g(t)+1=n$ and
\begin{align}\label{defi:inter-z2}
{\alpha_{n}\le \beta_{n-1}\le \alpha_{n-1}\le\cdots\le \beta_{2}\le \alpha_{2}\le \beta_{1}\le \alpha_{1}}.
\end{align}
We say that $g(t)$ interlaces $f(t)$ strictly, denoted $g(t)\prec f(t)$, if the inequalities in \eqref{defi:inter-z1} or \eqref{defi:inter-z2} hold strictly.
Following Liu and Wang \cite{LiuWang}, we also let $a\preceq bt+c$ for any nonnegative $a,b,c$, and let $0\preceq f$ and $f\preceq 0$ for any $f\in \mathrm{PF}$.
A polynomial sequence $\{f_n(t)\}_{n\geq 0}$ with each $f_n(t)\in \mathrm{PF}$ is said to be a \emph{generalized Sturm sequence} if
$${f_0(t)\preceq f_1(t)\preceq \cdots \preceq  f_{n-1}(t)\preceq f_n(t)\preceq \cdots.}$$
Liu and Wang \cite{LiuWang} established the following sufficient condition for determining whether two polynomials have interlacing zeros.

\begin{theorem}$($\cite[Theorem 2.3]{LiuWang}$)$
\label{thm:Liu-Wang}
Let $F(t),f(t),h_1(t),\ldots,h_k(t)$ be real polynomials satisfying the following conditions.
\begin{itemize}
\item[$(i)$]There exist some real polynomials $\phi(t),\psi_1(t),\ldots,\psi_k(t)$ such that
\begin{align}\label{eq-rec-LW}
F(t)=\phi(t)f(t)+\psi_1(t)h_1(t)+\cdots+\psi_k(t)h_k(t),
\end{align}
and $\deg F(t)= \deg f(t)$ or $\deg F(t)= \deg f(t)+1$.
\item[$(ii)$] $f(t),\,h_j(t)$ are polynomials with only real zeros, and $h_j(t)\preceq f(t)$ for all $j$.
\item[$(iii)$]The leading coefficients of $F(t), h_1(t),\ldots,h_k(t)$ have the same sign.
\end{itemize}
Suppose that $\psi_j(r)\le 0$ for each $j$ and each zero $r$ of $f(t)$. Then $F(t)$ has only real zeros and $f(t)\preceq F(t)$.
In particular, if for each zero $r$ of $f(t)$, there exists an index $j$ such that
$h_j(t)\prec f(t)$  and $\psi_j(r)<0$, then $f(t)\prec F(t)$.
\end{theorem}

Now we are able to prove the following result, which is stronger than Theorem \ref{thm:rz-lcv-1}.

\begin{theorem}\label{thm:intlac-zr-1}
For any given integer $d\ge 3$, the sequence $\{g_{n,d}(t)\}_{n\ge d+1}$ is a generalized Sturm sequence.
\end{theorem}

\begin{proof}
Use induction on $n$. By \eqref{g-uni-dn(t)} we have $g_{d+1,d}(t)=t$ and $g_{d+2,d}(t)=(d-1)t^2+dt$.
Clearly, both $g_{d+1,d}(t)$ and $g_{d+2,d}(t)$ are real-rooted polynomials, and
\begin{align}\label{prec:d1}
 g_{d+1,d}(t) \preceq g_{d+2,d}(t).
\end{align}
Assume that $g_{n,d}(t)$ and $g_{n+1,d}(t)$ are real-rooted polynomials and
$g_{n,d}(t) \preceq g_{n+1,d}(t)$. We proceed to  show that
\begin{align}\label{prec:dn+1}
 g_{n+1,d}(t) \preceq g_{n+2,d}(t).
\end{align}
For this purpose, we shall employ Theorem \ref{thm:Liu-Wang}. First, by recursion \eqref{rec:g-p-n1}, we have
\begin{align}\label{rec:g-p-n12}
 g_{n+2,d}(t)
=&\, \frac{(2d-n-1)t+n}{n-d+1} g_{n+1,d}(t)
 +\frac{t(n-d)}{n-d+1} g_{n,d}(t),\qquad n\ge d+1.
\end{align}
Observe that \eqref{rec:g-p-n12} is of the form of \eqref{eq-rec-LW} with $k=1$, where
$$F(t)=g_{n+2,d}(t),\quad f(t)=g_{n+1,d}(t),\quad h_1(t)=g_{n,d}(t),$$ and
\begin{align*}
\phi(t)=\frac{(2d-n-1)t+n}{n-d+1},\quad
\psi_1(t)=\frac{t(n-d)}{n-d+1}.
\end{align*}
By \eqref{g-uni-dn(t)}, we have that $\deg F(t)= \deg f(t)+1=n-d+2$ for $1\le n-d\le d-2$ and $\deg F(t)= \deg f(t)=d$ for $n-d\ge d-1$.
Thus, condition (i) of Theorem \ref{thm:Liu-Wang} is satisfied.
By the inductive hypothesis that $g_{n,d}(t) \preceq g_{n+1,d}(t)$, condition (ii) of Theorem \ref{thm:Liu-Wang} also holds.
Moreover, the leading coefficients of $F(t)$ and $h_1(t)$ are all positive by \eqref{g-uni-dn(t)}. Therefore, condition (iii) of Theorem \ref{thm:Liu-Wang} is also satisfied. It remains to show that $\psi_1(r)\le 0$ for each zero $r$ of $f(t)$. Note that $f(t)$ has only non-positive zeros since its coefficients are all non-negative. Clearly, for $t\le 0$ and $n\ge d+1$, we have $\psi_1(t)\le 0$.
From Theorem \ref{thm:Liu-Wang} it follows that \eqref{prec:dn+1} holds and $g_{n+2,d}(t)$ has only real zeros.
This completes the proof.
\end{proof}

It is interesting that the polynomials $g_{n,d}(t)$ also enjoy some interlacing property for fixed $n$. Precisely we have the following result,
which is also stronger than Theorem \ref{thm:rz-lcv-1}.

\begin{theorem}\label{thm:intlac-zr-2}
For any given integer $n\ge 2$, the sequence $\{g_{n,d}(t)\}_{d=1}^{[n/2]}$ is
a generalized Sturm sequence.
\end{theorem}

\begin{proof}
For $2\leq n\leq 5$, one can directly use \eqref{g-uni-dn(t)} to verify the assertion.
We may assume that $n\ge 6$. It remains to show that for $1\le d \le [n/2]-1$,
\begin{align}\label{intlc-d}
g_{n,d}(t)\preceq g_{n,d+1}(t).
\end{align}
We apply induction on $d$ to prove \eqref{intlc-d}. First, by \eqref{g-uni-dn(t)}, we get that $g_{n,1}(t)=t$ and $g_{n,2}(t)=(n-3)t^2+(n-2)t$.
It is clear that \eqref{intlc-d} holds for $d=1$. Assuming \eqref{intlc-d} for some $d \le [n/2]-2$, we proceed to show that
\begin{align}\label{intlc-d12}
g_{n,d+1}(t)\preceq g_{n,d+2}(t).
\end{align}

By the recurrence relation \eqref{rec:g-p-d1}, we have
\begin{align}\label{eq:recu-gnd+2}
 g_{n,d+2}(t)
= C_1 \cdot g_{n,d+1}(t) + C_2 \cdot g_{n,d}(t)
\end{align}
for $1\le d\le [n/2]-1$, where
\begin{align*}
C_1=&\, \frac{(n-2-2d)\big((n-3-2d)(n-1-2d)t+(n-d-2)^2+d^2+n-2\big)}
{(d+1)(n-2-d)(n-1-2d)},\\
C_2=&\,
 -\frac{d(n-1-d)(n-3-2d)}{(d+1)(n-2-d)(n-1-2d)}.
\end{align*}
Notice that \eqref{eq:recu-gnd+2} is of the form of \eqref{eq-rec-LW} with $k=1$, where $F(t)=g_{n,d+2}(t)$, $f(t)=g_{n,d+1}(t)$, $h_1(t)=g_{n,d}(t)$,
$\phi(t)=C_1$, and $\psi_1(t)=C_2$. The conditions (i), (ii), and (iii) of Theorem \ref{thm:Liu-Wang} are easily checked to be true for $1\le d \le [n/2]-2$. In order to prove \eqref{intlc-d12}, it remains to show $C_2<0$ for $1\le d \le [n/2]-2$, which is evident. This completes the proof.
\end{proof}

By applyig Theorem \ref{thm:gpn=2d}, we obtain three new interlacing properties for Speyer's $g$-polynomials.

\begin{theorem}\label{thm:intlc-n=2d}
The sequences $\{g_{2d,d}(t)\}_{d\ge 1}$, $\{g_{2d+1,d}(t)\}_{d\ge 1}$, and $\{g_{n,[n/2]}(t)\}_{n\ge 2}$ are generalized Sturm sequences. That is,
\begin{align}
g_{2d,d}(t) & \preceq g_{2d+2,d+1}(t),\quad d\ge 1, \label{inter:2d}\\
g_{2d+1,d}(t) & \preceq g_{2d+3,d+1}(t), \quad d\ge 1, \label{inter:2d+1}\\
g_{n,[n/2]}(t) & \preceq g_{n+1,[(n+1)/2]}(t), \quad n\ge 2. \label{inter:2n}
\end{align}
\end{theorem}

\begin{proof}

We first prove the interlacing property \eqref{inter:2d}.
For $d=1$, it is easily checked that $g_{2,1}(t)=t$ and $g_{4,2}(t)=t^2+2t$ by \eqref{g-uni-dn(t)}. Clearly, $g_{2,1}(t)\preceq g_{4,2}(t)$. Hence, we have \eqref{inter:2d} holds for $d=1$. Assume that \eqref{inter:2d} is true for $d\ge 1$. It suffices to show that
\begin{align}\label{inter:2d2}
g_{2d+2,d+1}(t) \preceq g_{2d+4,d+2}(t),\quad d\ge 1.
\end{align}
By the recurrence relation \eqref{rec:g-p-n2d} given in Theorem \ref{thm:gpn=2d} and the sufficient condition showed in Theorem \ref{thm:Liu-Wang}, we can derive the interlacing relation \eqref{inter:2d2} for $d\ge 1$ in a similar argument of the proof of Theorem \ref{thm:intlac-zr-1}, and the detailed proof is omitted here.

To prove the interlacing property \eqref{inter:2d+1}, observe by \eqref{g-uni-dn(t)} that $g_{3,1}(t)=t$, $g_{5,2}(t)=2t^2+3t$, and $g_{3,1}(t)\preceq g_{5,2}(t)$. Then in a similar argument as described above, one can obtain \eqref{inter:2d+1} by employing Theorem \ref{thm:Liu-Wang} and the recurrence relation \eqref{rec:g-p-n2do}.

We proceed to prove the  interlacing relation \eqref{inter:2n}. It suffices to show that for $d\ge 1$,
\begin{align}\label{inter:2d+12}
g_{2d,d}(t) \preceq g_{2d+1,d}(t) \preceq g_{2d+2,d+1}(t).
\end{align}
It is routine to check that $g_{2,1}(t)\preceq  g_{3,1}(t)\preceq g_{4,2}(t)\preceq  g_{5,2}(t)$ by using the explicit expressions of these four polynomials as shown above.
Thus, \eqref{inter:2d+12} holds for $d=1$. Assume that \eqref{inter:2d+12} holds for $d\ge 1$. We aim to show that
\begin{align}\label{inter:2d+123}
g_{2d+2,d+1}(t) \preceq g_{2d+3,d+1}(t) \preceq g_{2d+4,d+2}(t),
\end{align}
for $d\ge 1$.
By \eqref{rec:g-p-n2d+1} we have for $d\ge 1$,
$$
 g_{2d+3,d+1}(t)
=\frac{2d+1}{d+1} g_{2d+2,d+1}(t) + t\frac{d}{d+1} g_{2d+1,d}(t),
$$
which is of the form of \eqref{eq-rec-LW} with $k=1$. It is easy to verify that the conditions of Theorem \ref{thm:Liu-Wang} are satisfied, and hence
\begin{align}\label{inte:g2d23}
g_{2d+2,d+1}(t) \preceq g_{2d+3,d+1}(t).
\end{align}
Moreover, by \eqref{rec:g-p-n2d-1}, we have for $d\ge 1$,
$$
 g_{2d+4,d+2}(t)
=2\, g_{2d+3,d+1}(t) + t\, g_{2d+2,d+1}(t),
$$
which is of the form of \eqref{eq-rec-LW} with $k=1$.
In the same manner we get that
\begin{align}\label{inte:g2d34}
g_{2d+3,d+1}(t) \preceq g_{2d+4,d+2}(t).
\end{align}
Combining \eqref{inte:g2d23} and \eqref{inte:g2d34} yields \eqref{inter:2d+123}.
This completes the proof.
\end{proof}

\section{Asymptotic normality}\label{Sec:4}

This section is devoted to the study of asymptotic normality of the coefficients of Speyer's $g$-polynomial $g_{n,d}(t)$. 
The coefficients of many real-rooted polynomials appear to be asymptotically normal, see \cite{Bender, CMW-2020, CYZ-2022, CWZ-2020,LWW-2023} for examples.
Since each $g$-polynomial $g_{n,d}(t)$ has only real zeros, it is natural to investigate that whether Speyer's $g$-polynomials possess this asymptotic property.

Before showing the main result of this section, we first recall some concepts and results.
Let $\{f_n(t)\}_{n\geq 0}$ be a polynomial sequence with nonnegative coefficients, say
\begin{align}\label{def-f}
f_n(t)=\sum_{k=0}^{n}a(n,k)t^k.
\end{align}
The coefficient $a(n,k)$ is said to be asymptotically normal with mean $\mu_n$ and variance $\sigma_n^2$ by a central limit theorem if
\begin{align*}
\lim\limits_{n\rightarrow \infty} \sup\limits_{t\in \mathbb{R}}
\left| \sum\limits_{k\leq \mu_n+t\sigma_n} p(n,k) -\frac{1}{\sqrt{2\pi}}\int_{-\infty}^t \exp({-x^2/2}) dx \right|=0,
\end{align*}
where $p(n,k)=\frac{a(n,k)}{\sum_{j=0}^{n}a(n,j)}$.
The coefficient $a(n,k)$ is said to be asymptotically normal with mean $\mu_n$ and variance $\sigma_n^2$  by a local limit theorem on the real set $\mathbb{R}$ if
\begin{align*}
\lim\limits_{n\rightarrow \infty} \sup\limits_{t\in \mathbb{R}}
\left| \sigma_n p(n,\lfloor\mu_n+t \sigma_n\rfloor) -\frac{1}{\sqrt{2\pi}} \exp({-{t^2}/{2}}) \right|=0.
\end{align*}

The main result of this section is stated as follows.

\begin{theorem}\label{thm:asy-Sgp}
The coefficients of the Speyer's $g$-polynomial $g_{n,[n/2]}(t)$ is asymptotically normal by local and central limit theorems.
\end{theorem}

In order to prove Theorem \ref{thm:asy-Sgp}, we shall employ the following criterion provided by Bender \cite{Bender}. See also Harper \cite{Harper}.

\begin{theorem}[{\cite[Theorem 2]{Bender}}]\label{lemm-asymp-normal}
Let $\{f_n(t)\}_{n\geq 0}$ be a real-rooted polynomial sequence with nonnegative coefficients as in \eqref{def-f}.
Let
\begin{align}
\mu_n&=\frac{f_n'(1)}{f_n(1)},\label{eq-mean}\\[5pt]
\sigma_n^2&=\frac{f_n''(1)}{f_n(1)}+\mu_n-\mu_n^2.\label{eq-var}
\end{align}
If $\sigma_n^2\rightarrow +\infty$ as $n \rightarrow +\infty$, then the coefficient of $f_n(t)$ is asymptotically normal  with mean $\mu_n$ and variance $\sigma_n^2$  by local and central limit theorems.
\end{theorem}

Before proving Theorem \ref{thm:asy-Sgp}, let us first note several lemmas.

\begin{lemma}\label{lemma-1}
For $n\ge 3$ let $f_n(t):=g_{n,[n/2]}(t)$ and
\begin{align}\label{defi:rn(x)}
 r_n(t)=\frac{f_{n-1}(t)}{f_{n}(t)}.
\end{align}
Then \begin{align*}
 \lim_{n\rightarrow +\infty} r_n(1)=-1+\sqrt{2}.
\end{align*}
\end{lemma}

\begin{proof}
By \eqref{rec:g-p-n2d+1} and \eqref{rec:g-p-n2d-1} we obtain that, for $d\ge 1$,
\begin{align*}
f_{2d+1}(1)
=&\, \frac{2d-1}{d} f_{2d}(1) + \frac{d-1}{d} f_{2d-1}(1),\\
 f_{2d+2}(1)
=&\, 2\, f_{2d+1}(1) + f_{2d}(1),
\end{align*}
and hence,
\begin{align}\label{eq:fn+1nn-1}
 f_{n+1}(1)
=\frac{n-1}{\lfloor n/2 \rfloor} f_{n}(1)
 + \frac{\lceil n/2 \rceil-1}{\lfloor n/2 \rfloor} f_{n-1}(1),\quad n\ge 2,
\end{align}
where $\lfloor x \rfloor$ is the greatest integer not exceeding $x$ and $\lceil x \rceil$ is the least integer not less than $x$.
Suppose that
\begin{align*}
 \lim_{n\rightarrow +\infty} r_n(1)=A,
\end{align*}
where $A$ is a finite real number. It follows from \eqref{eq:fn+1nn-1} that
\begin{align*}
 A^{-1} = 2 + A,
\end{align*}
whose zeros are
$$
A_1=-1-\sqrt{2}, \quad A_2=-1+\sqrt{2}.
$$
Since $r_n(1)>0$ for all $n\ge 3$, we have $A=-1+\sqrt{2}$, as desired.
\end{proof}

The following lemma gives an explicit formula to compute the mean of the coefficients of
$g_{n,[n/2]}(t)$.

\begin{lemma}\label{lemma-2}
For $n\ge 3$ let $f_n(t):=g_{n,[n/2]}(t)$ and $\mu_n=\frac{f_n'(1)}{f_n(1)}$.
Then
\begin{align}\label{eq:mun}
 \mu_n
=\left(1-\frac{\lceil n/2 \rceil-1}{\lceil n/2 \rceil+1}\cdot
 \frac{f_{n-1}(1)}{f_n(1)}\right) \frac{\lceil n/2 \rceil+1}{2}.
\end{align}
\end{lemma}

\begin{proof}
By \eqref{rec:g-p-n2d-1'} and \eqref{rec:g-p-n2d+1'}, we find that
\begin{align}
 f_{2d}(1)
=&\, \frac{d-1}{d+1}\, f_{2d-1}(1)
 +\frac{2}{d+1}\, f_{2d}'(1), \quad d\ge 2,\ \textrm{and}\label{rec:f-n2d'1}\\
 f_{2d+1}(1)
=&\, \frac{d}{d+2}\, f_{2d}(1)
 +\frac{2}{d+2}\, f_{2d+1}'(1), \quad d\ge 1. \label{rec:f-n2d+1'1}
\end{align}
It follows from \eqref{rec:f-n2d'1} and \eqref{rec:f-n2d+1'1} that for $n\ge 3$,
\begin{align}\label{eq:fn'1/fn1}
 f_n(1) = \frac{\lceil n/2 \rceil-1}{\lceil n/2 \rceil+1}\, {f_{n-1}(1)}
     + \frac{2}{\lceil n/2 \rceil+1}\, f_n'(1).
\end{align}
By using the above formula, we immediately get \eqref{eq:mun}.
\end{proof}

With the above lemma, we are able to give an explicit formula to compute the variance of the coefficients of
$g_{n,[n/2]}(t)$.

\begin{lemma}\label{lem-3}
For $n\ge 3$ let $f_n(t):=g_{n,[n/2]}(t)$, $\mu_n=\frac{f_n'(1)}{f_n(1)}$ and
$\sigma_n^2=\frac{f_n''(1)}{f_n(1)}+\mu_n-\mu_n^2$.
Then
\begin{align}\label{eq:sigman2}
 \sigma_n^2
=&\, \frac{\lceil n/2 \rceil\left(\lceil n/2 \rceil - 1\right)}{4}
   - \frac{(\lceil n/2 \rceil -1)(2\lfloor n/2 \rfloor -1)}{4}
  \frac{f_{n-1}(1)}{f_{n}(1)} -\frac{\left(\lceil n/2 \rceil-1\right)^2}{4}
  \left(\frac{f_{n-1}(1)}{f_n(1)}\right)^2.
\end{align}
\end{lemma}

\begin{proof}
By \eqref{rec:g-p-n2d-1''} and \eqref{rec:g-p-n2d+1''}, we have
\begin{align*}
 f_{2d}(1)
=&\, \frac{4d-1}{2d+1}\, f_{2d-1}(1)
  +\frac{4}{(d-1)(2d+1)}\, f_{2d}''(1),\quad d\ge 2,\\
 f_{2d+1}(1)
=&\, \frac{4d+1}{2d+3}\, f_{2d}(1)
  +\frac{4}{d(2d+3)}\, f_{2d+1}''(1), \quad d\ge 1.
\end{align*}
It follows that
\begin{align}\label{rec:fn1fn''1}
 f_{n}(1)
=\frac{2n-1}{2 \lceil n/2 \rceil +1}\, f_{n-1}(1)
  +\frac{4}{(\lceil n/2 \rceil -1)(2 \lceil n/2 \rceil +1)}\, f_{n}''(1),\quad n\ge 3,
\end{align}
which can be rewritten as
\begin{align}\label{rec:fn''1/fn1}
 \frac{f_{n}''(1)}{f_{n}(1)}
=\left(1-\frac{2n-1}{2 \lceil n/2 \rceil +1}\, \frac{f_{n-1}(1)}{f_{n}(1)}\right)
  \frac{(\lceil n/2 \rceil -1)(2 \lceil n/2 \rceil +1)}{4},\, \quad n\ge 3.
\end{align}

By \eqref{eq:mun} and \eqref{rec:fn''1/fn1} we get
\begin{align*}
 \sigma_n^2
=&\, \left(1-\frac{2n-1}{2 \lceil n/2 \rceil +1}\,
  \frac{f_{n-1}(1)}{f_{n}(1)}\right)
  \frac{(\lceil n/2 \rceil -1)(2 \lceil n/2 \rceil +1)}{4}\\
 &\, +\left(1-\frac{\lceil n/2 \rceil-1}{\lceil n/2 \rceil+1}\cdot
 \frac{f_{n-1}(1)}{f_n(1)}\right) \frac{\lceil n/2 \rceil+1}{2}\\
 &\, -\left(1-\frac{\lceil n/2 \rceil-1}{\lceil n/2 \rceil+1}\cdot
 \frac{f_{n-1}(1)}{f_n(1)}\right)^2 \frac{(\lceil n/2 \rceil+1)^2}{4}.
\end{align*}
By further simplifying, we obtain the desired result.
\end{proof}

In order to study the limit of the variance of the coefficients of
$g_{n,[n/2]}(t)$ as $n\rightarrow +\infty$, we also need the following two lemmas.

\begin{lemma}\label{lem-4}
For $m\geq 2$, we have
\begin{align}\label{ieq:2mN1}
 m-(2m -1)r_{2m}(1)-(m-1) r_{2m}(1)^2> 2-\sqrt{2},
\end{align}
where $r_{2m}(1)={f_{2m-1}(1)}/{f_{2m}(1)}$ and $f_n(t):=g_{n,[n/2]}(t)$.
\end{lemma}

\begin{proof}
Let us first show that
\begin{align}\label{ieq:r2m1<up}
 r_{2m}(1)<-1+\sqrt{2}
\end{align}
holds for $m\ge 2$ by induction.

By \eqref{g-uni-dn(t)}, we have $r_{4}(1)={f_{3}(1)}/{f_{4}(1)}=1/3<-1+\sqrt{2}$. So, \eqref{ieq:r2m1<up} holds for $m=2$. Assume $r_{2m}(1)<-1+\sqrt{2}$. We proceed to show that $r_{2m+2}(1)<-1+\sqrt{2}$.
By \eqref{eq:fn+1nn-1}, we have for $m\ge 1$,
\begin{align*}
f_{2m+1}(1)
=&\, \frac{2m-1}{m} f_{2m}(1) + \frac{m-1}{m} f_{2m-1}(1),\\
 f_{2m+2}(1)
=&\, 2\, f_{2m+1}(1) + f_{2m}(1),
\end{align*}
or equivalently,
\begin{align*}
 \frac{1}{r_{2m+1}(1)}
=&\, 2-\frac{1}{m} + \frac{m-1}{m} r_{2m}(1),\\
 \frac{1}{r_{2m+2}(1)}
=&\, 2\, + r_{2m+1}(1).
\end{align*}
Combining the above two relations yields
\begin{align}\label{ieq:r2m21}
 r_{2m+2}(1)
= \frac{1}{2\, + r_{2m+1}(1)}
= \frac{1}{2\, + \frac{\displaystyle 1}{\displaystyle 2-\frac{1}{m} + \frac{m-1}{m} r_{2m}(1)}}.
\end{align}
By the inductive hypothesis and \eqref{ieq:r2m21}, we see for $m\ge 2$,
{\small{\begin{align*}
 r_{2m+2}(1)
&<\frac{1}{2\, + \frac{\displaystyle 1}{\displaystyle 2-\frac{1}{m} + \frac{m-1}{m} (-1+\sqrt{2})}}=\frac{1}{2\, + \frac{\displaystyle 1}{\displaystyle \sqrt{2}+1-\frac{\sqrt{2}}{m}}}
<\frac{1}{2\, + \frac{\displaystyle 1}{\displaystyle \sqrt{2}+1}}
=-1+\sqrt{2}.
\end{align*}}}
Thus, \eqref{ieq:r2m1<up} is proved. Hence for $m\ge 2$, there holds
\begin{align*}
  &\, m-(2m -1)r_{2m}(1)-(m-1) r_{2m}(1)^2\\
 >&\, m-(2m -1)(-1+\sqrt{2})-(m-1) (-1+\sqrt{2})^2\\
 =&\, 2-\sqrt{2},
\end{align*}
as desired. 
This completes the proof.
\end{proof}

\begin{lemma}\label{lem-5}
For $m\geq 3$, we have
\begin{align}\label{ieq:oddn>0}
 m+1-2m\, r_{2m+1}(1)-m\, r_{2m+1}(1)^2>0.
\end{align}
where $r_{2m+1}(1)={f_{2m}(1)}/{f_{2m+1}(1)}$ and $f_n(t):=g_{n,[n/2]}(t)$.
\end{lemma}

\begin{proof}
Note that, for $m\geq 3$, the inequality \eqref{ieq:oddn>0} holds if and only if
\begin{align}\label{ieq:dr2m1}
 -1-\frac{\sqrt{2m^2+m}}{m} < r_{2m+1}(1) < -1+\frac{\sqrt{2m^2+m}}{m}.
\end{align}
Clearly, the first inequality in \eqref{ieq:dr2m1} holds, since $r_{2m+1}(1)>0$ for $m\ge 1$. It remains to prove that
\begin{align}\label{ieq:r2m11ub}
 r_{2m+1}(1) < -1+\frac{\sqrt{2m^2+m}}{m}
\end{align}
for $m\geq 3$.
First, for $m=3$, one can use \eqref{defi:rn(x)} and \eqref{g-uni-dn(t)} to verify that
$$
 r_{7}(1)-\left(-1+\frac{\sqrt{21}}{3}\right)=\frac{38}{25}-\frac{1}{3}\sqrt{21}<0.
$$
Assuming \eqref{ieq:r2m11ub} for $m$, 
we continue to show that it also holds for $m+1$, namely
\begin{align*}
 r_{2m+3}(1) < -1+\frac{\sqrt{2(m+1)^2+m+1}}{m+1}.
\end{align*}
By the recurrence relation \eqref{eq:fn+1nn-1}, we have for $m\ge 1$,
\begin{align*}
 \frac{1}{r_{2m+3}(1)}
=&\, \frac{2m+1}{m+1} + \frac{m}{m+1} r_{2m+2}(1),\\
 \frac{1}{r_{2m+2}(1)}
=&\, 2\, + r_{2m+1}(1).
\end{align*}
Thus,
\begin{align*}
 {r_{2m+3}(1)}
=\frac{1}{\displaystyle \frac{2m+1}{m+1} + \frac{m}{m+1}\cdot \frac{1}{2\, + r_{2m+1}(1)}}.
\end{align*}
By the inductive hypothesis, we see that for $m\ge 1$,
\begin{align*}
 r_{2m+3}(1)
<\frac{1}{\displaystyle \frac{2m+1}{m+1} + \frac{m}{m+1}\cdot \frac{m}{m+\sqrt{2m^2+m}}}.
\end{align*}
It suffices to prove for $m\ge 3$,
\begin{align}\label{ieq:vry}
 \frac{1}{\displaystyle \frac{2m+1}{m+1} + \frac{m}{m+1}\cdot \frac{m}{m+\sqrt{2m^2+m}}}
< -1+\frac{\sqrt{2(m+1)^2+m+1}}{m+1}.
\end{align}
The inequality \eqref{ieq:vry} can be verified by the following {\tt Maple} command:
{\scriptsize
\begin{align*}
\hskip -27pt [>
\mathbf{evalf}\left(\mathbf{solve}\left(\left\{\frac{\displaystyle 1}{\displaystyle \frac{2m+1}{m+1} + \frac{m}{m+1}\cdot \frac{m}{m+\sqrt{2m^2+m}}}
+1-\frac{\sqrt{2(m+1)^2+m+1}}{m+1}<0,\ m>0\right\},m\right)\right)
\end{align*}
}%
It outputs
$$
\{1.393714057 < m\}
$$
which implies that \eqref{ieq:vry} holds for $m\ge 3$.
This completes the proof.
\end{proof}

We are now in a position to show a proof of Theorem \ref{thm:asy-Sgp}.

\begin{proof}[Proof of Theorem \ref{thm:asy-Sgp}]

For $n\ge 3$ let $f_n(t):=g_{n,[n/2]}(t)$, $\mu_n=\frac{f_n'(1)}{f_n(1)}$ and
$\sigma_n^2=\frac{f_n''(1)}{f_n(1)}+\mu_n-\mu_n^2$.
By Theorem \ref{lemm-asymp-normal} it suffices to show that $\sigma_n^2 \rightarrow +\infty$ as $n\rightarrow +\infty$.
To this end, we only need to show that
 $\lim_{m\rightarrow +\infty} \sigma_{2m}^2 = +\infty$ and
 $\lim_{m\rightarrow +\infty} \sigma_{2m+1}^2 = +\infty.$

By Lemmas \ref{lem-3} and \ref{lem-4} we see that for $m\ge 2$,
\begin{align}
 \sigma_{2m}^2
=\frac{m-1}{4}\big(m-(2m -1)r_{2m}(1)-(m-1) r_{2m}(1)^2\big)
> (2-\sqrt{2})\frac{m-1}{4},\label{simga2m2}
\end{align}
where $r_n(x)$ is defined by \eqref{defi:rn(x)}.
It follows from \eqref{simga2m2} that $\sigma_{2m}^2\rightarrow +\infty$ as $m\rightarrow +\infty$.

By Lemmas \ref{lem-3} and \ref{lem-5} we have
\begin{align*}
 \sigma_{2m+1}^2
=&\, \frac{m}{4}\big(m+1-(2m -1)r_{2m+1}(1)-m\, r_{2m+1}(1)^2\big)> \frac{m}{4} r_{2m+1}(1),
\end{align*}
for $m\ge 3$.
Since
$\lim_{m\rightarrow +\infty} r_{2m+1}(1)=-1+\sqrt{2}>0$ by Lemma \ref{lemma-1}, 
it follows that $\sigma_{2m+1}^2\rightarrow +\infty$ as $m\rightarrow +\infty$.
This completes the proof.
\end{proof}

We propose a conjecture to conclude this paper.
\begin{conjecture}
If $d=d(n)$ is a real function of $n$ such that $d(n)\rightarrow +\infty$ as $n\rightarrow +\infty$, then the coefficient of Speyer's $g$-polynomial $g_{n,d(n)}(t)$, namely $S_i(n,d(n))$, is asymptotically normal by local and central limit theorems.
\end{conjecture}

\section*{Acknowledgments}
The authors wish to thank Matthew H. Y. Xie for helpful comments and informing them about \cite{Ferroni-Schroter-2023}.


\begin{thebibliography}{alpha}

\bibitem{Bender}
E.A. Bender,
Central and local limit theorems applied to asymptotic enumeration,
J. Combin. Theory Ser. A  15 (1973), 91--111.

\bibitem{BEST-2023}
A. Berget, C. Eur, H. Spink, D. Tseng,
Tautological classes of matroids,
Invent. Math. 233 (2023), 951--1039.

\bibitem{ChenKauers}
S. Chen, M. Kauers,
Some open problems related to creative telescoping,
J. Syst. Sci. Complex, 30 (2017), 154--172.

\bibitem{CHM-2012}
W.Y.C. Chen, Q.-H. Hou, Y.-P. Mu,
The extended Zeilberger algorithm with parameters,
J. Symbolic Comput. 47(6) (2012), 643--654.

\bibitem{CMW-2020}
X. Chen, J. Mao, Y. Wang,
Asymptotic normality in $t$-stack sortable permutations,
Proc. Edinb. Math. Soc. (2), 63(4) (2020), 1062--1070.

\bibitem{CYZ-2022}
X. Chen, A.L.B. Yang, J.J.Y. Zhao,
Recurrences for Callan's generalization of Narayana polynomials,
J. Syst. Sci. Complex. 35 (2022), 1573--1585.

\bibitem{CWZ-2020}
X. Chen, Y. Wang, S.-N. Zheng,
Analytic properties of combinatorial triangles related to Motzkin numbers,
Discrete Math. 343 (2020), 112133.

\bibitem{Ferroni}
L. Ferroni,
Schubert matroids, Delannoy paths, and Speyer's invariant,
Combinatorial Theory, 3(3) (2023), No. 13, 22 pp.

\bibitem{Ferroni-Schroter-2023}
L. Ferroni, B. Schr\"{o}ter,
Valuative invariants for large classes of matroids,
J. Lond. Math. Soc, to appear.

\bibitem{FinkSpeyer2012}
A. Fink, D.E. Speyer,
$K$-classes for matroids and equivariant localization,
Duke Math. J. 161(14) (2012), 2699--2723.

\bibitem{HKT2006}
P. Hacking, S. Keel, E. Tevelev,
Compactification of the moduli space of hyperplane arrangements,
J. Algebraic Geom. 15 (2006), 657--680.

\bibitem{Harper}
L.H. Harper,
Stirling behavior is asymptotically normal,
Ann. Math. Statist. 38 (1967), 410--414.

\bibitem{APCI}
Q.-H. Hou,
Maple package APCI,
http://faculty.tju.edu.cn/HouQinghu/en/lwcg/4184/content/
      23455.htm\#lwcg.

\bibitem{Kapranov1993}
M. Kapranov,
Chow quotients of Grassmannians I,
Adv. Soviet Math. 16(2) (1993), 29--110.

\bibitem{LWW-2023}
H. Liang, Y. Wang, Y. Wang,
Analytic aspects of generalized central trinomial coefficients,
J. Math. Anal. Appl. 527(1) (2023), 127424.

\bibitem{LiuWang}
L.L. Liu, and Y. Wang,
A unified approach to polynomoal sequences with only real zeros,
Adv. Appl. Math. 38 (2007), 542--560.

\bibitem{LRS2020}
L. L\'{o}pez de Medrano, F. Rinc\'{o}n, K. Shaw,
Chern-Schwartz-MacPherson cycles of matroids,
Proc. Lond. Math. Soc. (3) 120(1) (2020),  1--27.

\bibitem{Speyer2008}
D.E. Speyer,
Tropical linear spaces,
SIAM J. Discrete Math. 22(4) (2008), 1527--1558.

\bibitem{Speyer2009}
D.E. Speyer,
A matroid invariant via the $K$-theory of the Grassmannian,
Adv. Math. 221(3) (2009) 882--913.

\bibitem{WZ1992}
H.S. Wilf, D. Zeilberger,
An algorithmic proof theory for hypergeometric (ordinary and ``$q$'') multisum/integral identities,
Invent. Math. 108 (1992), 575--633.

\bibitem{Zeilberger1991}
D. Zeilberger,
The method of creative telescoping,
J. Symbolic Comput. 11 (1991), 195--204.

\end{thebibliography}
\end{document}